\documentclass{amsart}

\usepackage{esint}
\usepackage{amsrefs}
\usepackage[toc,page]{appendix}
\usepackage{amssymb}
\usepackage{combelow}

\usepackage{tikz}
\usetikzlibrary{positioning}

\newtheorem{theorem}{Theorem}[section]
\newtheorem{lemma}[theorem]{Lemma}

\theoremstyle{definition}

\newtheorem{openproblem}[theorem]{Open Problem}

\theoremstyle{remark}
\newtheorem{remark}[theorem]{Remark}

\numberwithin{equation}{section}

\begin{document}

\title{A Noninequality for the Fractional Gradient}


\author{Daniel Spector}
\address{
Daniel Spector\hfill\break\indent
National Chiao Tung University\hfill\break\indent
Department of Applied Mathematics\hfill\break\indent
Hsinchu, Taiwan, R.O.C.}
\address{Nonlinear Analysis Unit\hfill\break\indent 
Okinawa Institute of Science and Technology Graduate University\hfill\break\indent
1919-1 Tancha, Onna-son, Kunigami-gun\hfill\break\indent 
Okinawa, Japan}
\curraddr{}
\email{dspector@math.nctu.edu.tw}
\thanks{Supported in part by Taiwan Ministry of Science and Technology Research Grant 107-2115-M-009-002-MY2.}


\subjclass[2010]{Primary }

\date{}

\dedicatory{}

\commby{}

\begin{abstract}
In this paper we give a streamlined proof of an inequality recently obtained by the author:  For every $\alpha \in (0,1)$ there exists a constant $C=C(\alpha,d)>0$ such that
 \begin{align*}
\|u\|_{L^{d/(d-\alpha),1}(\mathbb{R}^d)} \leq C \| D^\alpha u\|_{L^1(\mathbb{R}^d;\mathbb{R}^d)}
\end{align*}
for all $u \in L^q(\mathbb{R}^d)$ for some $1 \leq q<d/(1-\alpha)$ such that $D^\alpha u:=\nabla I_{1-\alpha} u \in L^1(\mathbb{R}^d;\mathbb{R}^d)$.  We also give a counterexample which shows that in contrast to the case $\alpha =1$, the fractional gradient does not admit an $L^1$ trace inequality, i.e. $\| D^\alpha u\|_{L^1(\mathbb{R}^d;\mathbb{R}^d)}$ cannot control the integral of $u$ with respect to the Hausdorff content $\mathcal{H}^{d-\alpha}_\infty$.  The main substance of this counterexample is a result of interest in its own right, that even a weak-type estimate for the Riesz transforms fails on the space $L^1(\mathcal{H}^{d-\beta}_\infty)$, $\beta \in [1,d)$.  It is an open question whether this failure of a weak-type estimate for the Riesz transforms extends to $\beta \in (0,1)$.
\end{abstract}

\maketitle

\section{Introduction}
Let $\alpha \in (0,d)$ and define the Riesz potential of order $\alpha$ by
\begin{align*}
I_\alpha f(x) := \frac{1}{\gamma(\alpha)} \int_{\mathbb{R}^d} \frac{f(y)}{|x-y|^{d-\alpha}}\;dy,
\end{align*}
where $\gamma(\alpha)=\pi^{d/2}2^\alpha\Gamma(\alpha/2)/\Gamma(d/2-\alpha/2)$.  The failure of the possibility of the inequality
\begin{align}\label{false}
\|I_\alpha f \|_{L^{d/(d-\alpha)}(\mathbb{R}^d)} \leq C \|f\|_{L^1(\mathbb{R}^d)}
\end{align}
to hold for all $f \in L^1(\mathbb{R}^d)$ is a classical result in harmonic analysis (see, e.g. p.~119 in \cite{Sharmonic}).  This has led to various replacements, for example a weak-type estimate which has been pioneered by A. Zygmund \cite{Zygmund}:  For $\alpha \in (0,d)$ there exists a constant $C=C(\alpha,d)>0$ such that
\begin{align}\label{replacement1}
 |\{|I_\alpha f|>t\}|^{(d-\alpha)/d} \leq \frac{C}{t} \|f\|_{L^1(\mathbb{R}^d)}
\end{align}
for all $t>0$ and all $f \in L^1(\mathbb{R}^d)$.  With more assumptions one can obtain an estimate in the correct scaling Lebesgue space, for example the following consequence of the Hardy space estimate of E. Stein and G. Weiss \cite{SteinWeiss}:  For $\alpha \in (0,d)$ there exists a constant $C=C(\alpha,d)>0$ such that
\begin{align}\label{replacement2}
\|I_\alpha f \|_{L^{d/(d-\alpha)}(\mathbb{R}^d)} \leq C\left( \|f\|_{L^1(\mathbb{R}^d)} + \| DI_1f\|_{L^1(\mathbb{R}^d;\mathbb{R}^d)}\right)
\end{align}
for all $f \in L^1(\mathbb{R}^d)$ such that $DI_1 f \in L^1(\mathbb{R}^d;\mathbb{R}^d)$ (which is to say $f \in \mathcal{H}^1(\mathbb{R}^d)$, the real Hardy space).

While the estimate \eqref{replacement1} is optimal on this scale of spaces - which are the natural spaces to consider when one takes into account the integrability of the fundamental solution to the associated differential equation - the estimate \eqref{replacement2} admits improvements.  In particular, it was first observed by A. Schikorra, the author, and J. Van Schaftingen in \cite[Theorem A']{SSVS} that one does not need the $L^1(\mathbb{R}^d)$-norm of $f$:   Let $d\geq 2$ and $\alpha \in (0,d)$.  There exists a constant $C=C(\alpha,d)>0$ such that
\begin{align}\label{replacement2improvement}
\|u \|_{L^{d/(d-\alpha)}(\mathbb{R}^d)} \leq C \| D^\alpha  u\|_{L^1(\mathbb{R}^d;\mathbb{R}^d)}
\end{align}
for all $u \in C^\infty_c(\mathbb{R}^d)$ such that $D^\alpha u:= \nabla I_{1-\alpha} u \in L^1(\mathbb{R}^d;\mathbb{R}^d)$.  It was then subsequently proved by the author in \cite[p.~16]{Spector} that one has the optimal inequality on the Lorentz scale:  Let $d\geq 2$ and $\alpha \in (0,d)$.  There exists a constant $C=C(\alpha,d)>0$ such that
\begin{align}\label{replacement2improved}
\|u \|_{L^{d/(d-\alpha),1}(\mathbb{R}^d)} \leq C \| D^\alpha  u\|_{L^1(\mathbb{R}^d;\mathbb{R}^d)}
\end{align}
for all $u \in L^q(\mathbb{R}^d)$ for some $1 \leq q < \frac{d}{1-\alpha}$ such that $D^\alpha u \in L^1(\mathbb{R}^d;\mathbb{R}^d)$.  Here we recall that for such functions the fractional gradient $D^\alpha$ is defined distributionally as
\begin{align*}
\left \langle D^\alpha u , \Phi \right \rangle := - \int_{\mathbb{R}^d} I_{1-\alpha} u(x) \operatorname*{div}\Phi(x)\;dx
\end{align*}
for $\Phi \in C^\infty_c(\mathbb{R}^d)$.  Such a definition makes sense by the hypothesis $u \in L^q(\mathbb{R}^d)$ for some $1\leq q <d/(1-\alpha)$ as this ensures $I_{1-\alpha}u \in L^1_{loc}(\mathbb{R}^d)$.  One can compare \eqref{replacement2improvement} and \eqref{replacement2improved} with \eqref{replacement2}, at least for functions in the Hardy space, by taking $u=I_\alpha f$, as under this hypothesis that $D^\alpha I_\alpha f = D I_1 f$ as $L^1(\mathbb{R}^d;\mathbb{R}^d)$ vector functions is justified by the computation
\begin{align*}
\left \langle D^\alpha I_\alpha f , \Phi \right \rangle &= -\int_{\mathbb{R}^d} I_{1-\alpha} I_\alpha f(x) \operatorname*{div}\Phi(x)\;dx
\\
&=-\int_{\mathbb{R}^d}  I_1f(x) \operatorname*{div}\Phi(x)\;dx
\\
&= \int_{\mathbb{R}^d}  DI_1f(x) \cdot \Phi(x)\;dx.
\end{align*}

The choice to use $D^\alpha$ as an intrinsic object, in contrast to the classically studied fractional Laplacian, Riesz potentials, and Riesz transform is motivated by its analogy with the gradient, which has been studied, for example, in the calculus of variations \cite{SS,SS1,Schikorra1,Schikorra2}, in partial differential equations \cite{SSS,SSS1, ABES}, in relation to the theory of functions of bounded variation in \cite{BLNT, Comi-Stefani}, in continuum mechanics in \cite{Silhavy}, and in the Hardy and Sobolev inequalities established in \cite{SS1,SSVS,Spector}.  Notably there are places where the two differ, as for example in the work of G. Comi and G. Stefani \cite[Theorem 3.11 and Corollary 5.6]{Comi-Stefani}, where among other results they prove that the fractional gradient does not admit a coarea formula.  

The purpose of this paper is two-fold:  First, we give a streamlined proof of \eqref{replacement2improved} which removes some of the technical aspects concerning Lorentz spaces and sheds some insight into the estimate; second, we provide another example of a place where the fractional gradient and gradient diverge in the form of a noninequality for the fractional gradient - its failure to control the integral of $u$ with respect to an appropriate Hausdorff content.  In fact, our proof of this noninequality rests on a result of independent interest, which is a failure of even a weak-type bound for the Riesz transforms on the space of functions which are integrable with respect to the Hausdorff content.

Thus let us begin with a more transparent proof of the inequality \eqref{replacement2improved}.  In particular, while the proof in \cite{Spector} relied only on H\"older's inequality, the use of equivalent quasi-norms, and scaling properties of the Lorentz spaces, we here remove the reliance on any of their particular properties aside from the definition of the quasi-norm on $L^{d/(d-\alpha),1}(\mathbb{R}^d)$:
\begin{align*}
\|g \|_{L^{d/(d-\alpha),1}(\mathbb{R}^d)}:= \int_0^\infty |\{ |g|>t\}|^{(d-\alpha)/d}\;dt.
\end{align*}
The perspective we develop here begins with a lemma proved by S. Krantz, M. Peloso, and the author in \cite{KrantzPelosoSpector}, where an extension of the inequality \eqref{replacement2improved} to the setting of stratified groups has been proved.  The projection of this result in Euclidean space is
\begin{lemma}\label{lemma1}
Let $\alpha \in (0,1)$.  There exists a constant $C=C(\alpha,d)>0$ such that for $\mathcal{L}^d$ almost every $x \in \mathbb{R}^d$ one has the inequality
\begin{align}\label{interpolation}
|I_\alpha D\chi_E| \leq C \left(\sup_{t>0} |p_t \ast D \chi_E|\right)^{1-\alpha} \left(\sup_{t>0} |t^{1/2}Dp_t \ast \chi_E|\right)^{\alpha}
\end{align}
for all $\chi_E \in BV(\mathbb{R}^d)$, where
\begin{align*}
p_t(x)=\frac{1}{(4\pi t)^{d/2}} e^{\frac{-|x|^2}{4t}}
\end{align*} 
denotes the heat kernel on $\mathbb{R}^d$.  Here we use $BV(\mathbb{R}^d)$ to denote the set of $u \in L^1(\mathbb{R}^d)$ such that the distributional derivative $Du$ is a vector-valued Radon measure with finite total mass.
\end{lemma}

This a pointwise interpolation inequality in the spirit of that of V. Maz'ya and T. Shaposhnikova established in \cite[Lemma at p.~114]{MS} (see also Section 3 in \cite[Lemma 3.2]{Spector}), though its validity for functions of bounded variation avoids the technical difficulties of the fine properties of such functions one should pay attention to with the use of Hardy's inequality and the Hardy-Littlewood maximal function.

It is here that we deviate from the argument of \cite{Spector} and \cite{KrantzPelosoSpector}, in that from the inequality \eqref{interpolation} we extract two further inequalities:
\begin{align}
|I_\alpha D\chi_E| &\leq C \left(\sup_{t>0} |p_t \ast D \chi_E|\right)^{1-\alpha}  \label{local}\\
|I_\alpha D\chi_E| &\leq C \left(\sup_{t>0} |p_t \ast D \chi_E|+\sup_{t>0} |t^{1/2}Dp_t \ast \chi_E|\right). \label{global}
\end{align}
The former follows from inequality
\begin{align*}
|t^{1/2}Dp_t \ast \chi_E|(x) \leq \|t^{1/2}Dp_t\|_{L^1(\mathbb{R}^d;\mathbb{R}^d)} \|\chi_E\|_{L^\infty(\mathbb{R}^d)}
\end{align*}
and the fact that $\|t^{1/2}Dp_t\|_{L^1(\mathbb{R}^d;\mathbb{R}^d)}$ is bounded uniformly in $t$, while the latter is just Young's inequality.  

The key insight one gains from this splitting is that if one works directly with the Lorentz quasi-norm, the estimate \eqref{local} should be utilized for large values of $t$, which corresponds to $x$ near the boundary of $E$.  Meanwhile, the estimate \eqref{global} should be utilized for small values of $t$, which corresponds to $x$ far away (from the boundary of $E$).  Putting these two estimates together we obtain
\begin{lemma}\label{lemma2}
Let $d\geq 2$ and $\alpha \in (0,1)$.  There exists a constant $C=C(\alpha,d)>0$ such that
\begin{align*}
\|I_\alpha D\chi_E\|_{L^{d/(d-\alpha),1}(\mathbb{R}^d;\mathbb{R}^d)} \leq C Per(E)^{1-\alpha} |E|^{\alpha(1-1/d)}.
\end{align*}
for all $\chi_E \in BV(\mathbb{R}^d)$, where
\begin{align*}
Per(E):=|D\chi_E|(\mathbb{R}^d).
\end{align*}
\end{lemma}
The deduction of \eqref{replacement2improved} then proceeds as in \cite{Spector}, which amounts to firstly apply the coarea formula and the classical isoperimetric inequality, as in the classical work of H. Federer and W. Fleming \cite[Remark 6.6 on p.~487]{FF} and V. Maz'ya \cite[Equation (6) and Theorem 6]{mazya1960} (see also the more recent exploration by V. Maz'ya of the coarea formula in such inequalities \cite[Corollary 6.2 and Example 6.4; Theorem 7.1 and Remark 7.2]{mazya}), and secondly to invoke the boundedness of the Riesz transforms on $L^{p,q}(\mathbb{R}^d)$ for $1<p<+\infty$ and $1\leq q \leq +\infty$.    

Returning to the second purpose of this paper, let us recall the historical progression of Sobolev inequalities in the $L^1$ regime.  The most classical of these results is due to E. Gagliardo \cite{Gagliardo} and L. Nirenberg \cite{Nirenberg}, whose work shows the validity of the inequality
\begin{align}\label{gn}
\|u \|_{L^{d/(d-1)}(\mathbb{R}^d)} \leq C  \int_{\mathbb{R}^d} |\nabla u(x) | \;dx
\end{align}
for all $u \in W^{1,1}(\mathbb{R}^d)$.  A strengthening of this inequality on the Lorentz scale was subsequently obtained by A. Alvino \cite{alvino}, who proved the inequality
\begin{align}\label{aa}
\|u \|_{L^{d/(d-1),1}(\mathbb{R}^d)} \leq C \int_{\mathbb{R}^d} |\nabla u(x) | \;dx
\end{align} 
for all $u \in W^{1,1}(\mathbb{R}^d)$.  At approximately the same time, N. Meyer and W.P. Ziemer \cite{Meyers-Ziemer-1977} proved an inequality which contains \eqref{gn}, \eqref{aa}, and even Hardy's inequality
\begin{align}\label{hardy}
\int_{\mathbb{R}^d} \frac{|u(x)|}{|x|}\;dx \leq C  \int_{\mathbb{R}^d} |\nabla u(x) | \;dx
\end{align}
for all $u \in W^{1,1}(\mathbb{R}^d)$.  Precisely, in \cite[Theorem 4.7]{Meyers-Ziemer-1977} they proved the validity of the inequality
\begin{align}\label{replacement2optimal}
\int_{\mathbb{R}^d} |u^*|\;d\mu \leq C  \int_{\mathbb{R}^d} |\nabla u | \;dx\end{align}
for all $u \in W^{1,1}(\mathbb{R}^d)$ (and even $BV(\mathbb{R}^d)$) and all non-negative Radon measures $\mu$ such that $\mu(B(x,r)) \leq C' r^{d-1}$ for all $x \in \mathbb{R}^d$ and $r>0$, for some constant $C'>0$.  Here we use $u^*$ to denote the precise representative of $u$,
\begin{align*}
u^*(x):= \lim_{\epsilon \to 0} \fint_{B(x,\epsilon)} u(y)\;dy,
\end{align*}
which is defined up to a set of $\mathcal{H}^{d-1}$ measure zero, which follows by D. Adams' strong-type estimate for the Hardy-Littlewood maximal function \cite{Adams:1988}.

As was observed in \cite{PonceSpector2,Spector}, various choices of $\mu$ yield \eqref{gn}, \eqref{aa}, and \eqref{hardy}, so that \eqref{replacement2optimal} is a sort of master inequality.  The other implications are as follows.  The Lorentz space estimate \eqref{aa} implies \eqref{hardy} by H\"older's inequality on the Lorentz scale, while one deduces \eqref{gn} from \eqref{aa} via the inequality
\begin{align*}
\|u \|_{L^{d/(d-1)}(\mathbb{R}^d)} \leq C \|u \|_{L^{d/(d-1),1}(\mathbb{R}^d)}.
\end{align*}
Meanwhile the left-hand-side of \eqref{aa} is equivalent to the left-hand-side of \eqref{hardy} for non-negative, radially decreasing functions (see e.g. Lemma 4.3 in \cite{FrankSeiringer}), and so if one assumes a Poly\'a-Szeg\"o inequality has been established then a symmetrization argument shows \eqref{hardy} implies \eqref{aa}.  Finally, if one assumes a coarea formula has been established then \eqref{gn} implies \eqref{aa}.  This can be summarized as the following graphic:

\begin{tikzpicture}
\matrix [column sep=7mm, row sep=5mm] {
  \node (se)  {}; &
  \node (yw) [draw, shape=rectangle] {Trace Inequality}; &
  \node (ul)  {}; \\
  &  \node (ec) [draw, shape=rectangle] {Lorentz Space Inequality};&
  \node (d2) {}; \\
  \node (pu) [draw, shape=rectangle] {Lesbesgue Scale Inequality};&
   \node (d1) {}; &
  \node (we) [draw, shape=rectangle] {Hardy's Inequality};   \\
};

\draw[->, thick] (yw) edge[bend left] (we);
\draw[->, thick] (ec) -- (pu);
\draw[->, thick] (yw) edge[bend right] (pu);
\draw[->, thick] (yw) --(ec);
\draw[->, thick] (ec) --(we);
\draw[->, thick] (we) edge[bend left]  node[sloped, anchor=center, below, text width=4.0cm] { Poly\'a-Szeg\"o Inequality}(ec);
\draw[->, thick] (pu) edge[bend right]  node[sloped, anchor=center, below, text width=3.0cm] { Coarea Formula} (ec);
\end{tikzpicture}

We can now compare the known inequalities in the fractional regime.  The analogue of \eqref{gn} is \eqref{replacement2improvement}, proved in \cite{SSVS} - while the inequality is proved for $C^\infty_c(\mathbb{R}^d)$ functions, the definition of the fractional gradient shows such functions are curl free in the sense of distributions and therefore A. Bonami and S. Poornima's approximation argument \cite{BonamiPoornima} can be used to deduce the general case (see the argument on p.~16 in \cite{Spector}).  The analogue of \eqref{aa} is \eqref{replacement2improved}, proved in \cite{Spector}.  In particular, in Theorem 1.4 in \cite{Spector}, the inequality \eqref{replacement2improved} is established in the argument of the analogue for the fractional gradient of the inequality \eqref{hardy}.  Thus, a natural question is whether one has a stronger inequality in an analogue of \eqref{replacement2optimal}.  An answer to this question in the negative is given in
\begin{theorem}\label{cor}
Let $\alpha \in (0,1)$.  There is no universal constant $C=C(\alpha,d)>0$ such that
\begin{align}\label{false}
\int_{\mathbb{R}^d} |u|\;d\mu \leq C \int_{\mathbb{R}^d} |D^\alpha u|
\end{align}
for all $u \in L^q(\mathbb{R}^d) \cap C(\mathbb{R}^d)$ for some $1\leq q<d/(1-\alpha)$ such that $D^\alpha u \in L^1(\mathbb{R}^d;\mathbb{R}^d)$ and all non-negative Radon measures $\mu$ such that $\mu(B(x,r)) \leq C' r^{d-\alpha}$ for all $x \in \mathbb{R}^d$ and $r>0$, for some constant $C'>0$.
\end{theorem}
\begin{remark}
In the recent paper of F. Gmeineder, B. Rai\cb{t}\v{a}, and J. Van Schaftingen \cite{GRV}, certain weaker trace inequalities have been established for a wide class of first order linear homogeneous differential operators.  As a consequence of Theorem \ref{cor} one deduces the impossibility of an improvement of their results to the optimal result known for the gradient for $\alpha \in (0,1)$.
\end{remark}

As developed in the work of D. Adams \cite[Proposition 1 on p.~118]{Adams:1988}, the validity the inequality \eqref{false} is equivalent to the validity of
\begin{align*}
\int_0^\infty \mathcal{H}^{d-\alpha}_\infty\left(\{ |u|>t\}\right)\;dt \leq C \int_{\mathbb{R}^d} |D^\alpha u|
\end{align*}
for all $u \in L^q(\mathbb{R}^d) \cap C(\mathbb{R}^d)$ such that $D^\alpha u \in L^1(\mathbb{R}^d;\mathbb{R}^d)$, where $\mathcal{H}^{d-\alpha}_\infty$ is the Hausdorff content, defined for compact sets $K \subset \mathbb{R}^d$ by
\begin{align}\label{hausdorffcontent}
\mathcal{H}^{d - \alpha}_\infty(K) 
:= \inf\biggl\{ \sum_{i=0}^\infty{\omega_{d-\alpha}r_i^{d - \alpha}} : K \subset \bigcup_{i=0}^\infty B(x_i,r_i) \biggr\},
\end{align}
where $\omega_{d-\alpha} := {\pi^{(d-\alpha)/2}}/{\Gamma\left(\frac{d-\alpha}{2}+1\right)}$.  For general sets one then extends by regularity (see \cite[Section 2]{PonceSpector2}).  In particular, Theorem \ref{cor} shows that the fractional gradient does not admit a trace inequality/Hausdorff content estimate.

At the heart of the construction in our proof of Theorem \ref{cor} one can observe a phenomena of independent interest, which is the lack of a weak-type bound for the Riesz transform with respect to the Hausdorff content.  Here we recall that in \cite{Adams:1988} D. Adams proved that for $\beta \in (0,d)$, the Hardy-Littlewood maximal function is bounded on the space $L^1(\mathcal{H}^{d-\beta}_\infty)$, i.e. there exists a constant $C>0$ such that the inequality
\begin{align*}
\int_0^\infty \mathcal{H}^{d-\beta}_\infty\left(\{ \mathcal{M}(u)>t\}\right)\;dt \leq C \int_0^\infty \mathcal{H}^{d-\beta}_\infty\left(\{ |u|>t\}\right)\;dt 
\end{align*}
holds for all $u \in L^1(\mathcal{H}^{d-\beta}_\infty)$ (which is defined as the completion of continuous functions with compact support with respect to the functional on the right hand side of the preceding).  One might wonder whether a similar inequality holds for the Riesz transform, or even the weaker estimate
\begin{align}\label{weakRiesz}
\mathcal{H}^{d-\beta}_\infty(\{ |Ru|>t\})  \leq \frac{C}{t} \int_0^\infty \mathcal{H}^{d-\beta}_\infty\left(\{ |u|>t\}\right)\;dt,
\end{align}
where $Ru=DI_1u$ is the vector Riesz transform.  When $\beta \in [1,d)$, we obtain that no such inequality is possible as a consequence of
\begin{lemma}\label{thm2}
Let $\beta \in [1,d)$.  If $Q$ denotes the cube $[0,1]^d$, then
\begin{align*}
\sup_{t>0} t \mathcal{H}^{d-\beta}_\infty(\{ |D I_1\chi_Q|>t\}) = +\infty.
\end{align*}
\end{lemma}
Indeed, $\chi_Q \in L^1(\mathcal{H}^{d-\beta}_\infty)$ for any $\beta \in (0,d)$, while the preceding shows that the weak-type quasi-norm of its Riesz transform is unbounded for $\beta \in [1,d)$ and so for such $\beta$ one cannot have \eqref{weakRiesz}.  This motivates
\begin{openproblem}
For $\beta \in (0,1)$, can one find $u \in L^1(\mathcal{H}^{d-\beta}_\infty)$ such that
\begin{align*}
\sup_{t>0} t \mathcal{H}^{d-\beta}_\infty(\{ |D I_1u|>t\}) = +\infty?
\end{align*}
\end{openproblem}

In the next section we prove the main results of this paper.  We first prove Lemma \ref{lemma1} and Lemma \ref{lemma2}.  We then proceed to prove Lemma \ref{thm2}, before returning to prove Theorem \ref{cor}, as the details in the construction of the former will be useful in proving the latter.

\section{Proofs of the Main Results}\label{mainresults}
We begin this section with the proof of Lemma \ref{lemma1}.  

\begin{proof}[Proof of Lemma \ref{lemma1}]
We first claim that for Lebesgue almost every $x \in \mathbb{R}^d$ one has the equality
\begin{align}\label{gaussianrepresentation}
I_\alpha D\chi_E(x) &= \frac{1}{\Gamma(\alpha/2)} \int_0^\infty t^{\alpha/2-1} p_t \ast D\chi_E \;dt.
\end{align}

Indeed, one can check by Fourier transform that for $\varphi \in C^\infty_c(\mathbb{R}^d)$ one has
\begin{align*}
I_\alpha \varphi (x) = \frac{1}{\Gamma(\alpha/2)} \int_0^\infty t^{\alpha/2-1} p_t \ast \varphi \;dt.
\end{align*}
From this and several applications of Fubini's theorem we find
\begin{align*}
\int_{\mathbb{R}^d}  I_\alpha D\chi_E \;\varphi \;dx &=  \frac{1}{\Gamma(\alpha/2)} \int_{\mathbb{R}^d}  \int_0^\infty t^{\alpha/2-1} p_t \ast \varphi \;dt\;dD\chi_E \\
&=\frac{1}{\Gamma(\alpha/2)}\int_0^\infty t^{\alpha/2-1}  \int_{\mathbb{R}^d}   p_t \ast D\chi_E  \;\varphi \;dx \;dt \\
&= \frac{1}{\Gamma(\alpha/2)} \int_{\mathbb{R}^d} \int_0^\infty t^{\alpha/2-1}    p_t \ast D\chi_E   \;dt  \;\varphi \;dx.
\end{align*}
As this equality holds for all $\varphi \in C^\infty_c(\mathbb{R}^d)$, we deduce the equality \eqref{gaussianrepresentation}.

We now use the representation \eqref{gaussianrepresentation} to estimate as follows
\begin{align*}
I_\alpha D\chi_E(x) &= \frac{1}{\Gamma(\alpha/2)} \int_0^\infty t^{\alpha/2-1} p_t \ast D\chi_E \;dt \\
&= \frac{1}{\Gamma(\alpha/2)} \int_0^r t^{\alpha/2-1} p_t \ast D\chi_E \;dt+\frac{1}{\Gamma(\alpha/2)} \int_r^\infty t^{\alpha/2-1} p_t \ast D\chi_E \;dt\\
&=: I(r)+II(r).
\end{align*}
For $I(r)$, we have
\begin{align*}
|I(r)| &= \left|\frac{1}{\Gamma(\alpha/2)} \int_0^r t^{\alpha/2-1} p_t \ast D\chi_E \;dt\right|\\
&\leq \frac{1}{\Gamma(\alpha/2)} \sup_{t>0} |p_t \ast D\chi_E| \int_0^r t^{\alpha/2-1}  \;dt \\
&= \frac{1}{\Gamma(\alpha/2+1)} r^{\alpha/2} \sup_{t>0} |p_t \ast D\chi_E|.
\end{align*}
Meanwhile, for $II(r)$, we integrate by parts in the convolution to obtain
\begin{align*}
|II(r)| &= \left|\frac{1}{\Gamma(\alpha/2)} \int_r^\infty t^{\alpha/2-1} Dp_t \ast \chi_E \;dt \right| \\
&\leq \frac{1}{\Gamma(\alpha/2)} \sup_{t>0} |t^{1/2}Dp_t \ast \chi_E| \int_r^\infty t^{\alpha/2-1/2-1}  \;dt \\
&=\frac{1}{\Gamma(\alpha/2)} \sup_{t>0} |t^{1/2}Dp_t \ast \chi_E|  \frac{r^{\alpha/2-1/2}}{1/2-\alpha/2}
\end{align*}
One can then optimize in $r$, though setting the upper bounds for $I(r)$ and $II(r)$ is sufficient for our purposes, from which we obtain
\begin{align*}
|I_\alpha D\chi_E(x)| \leq C \left(\sup_{t>0} |p_t \ast D\chi_E| \right)^{1-\alpha} \left(\sup_{t>0} |t^{1/2}Dp_t \ast \chi_E| \right)^\alpha
\end{align*}
with
\begin{align*}
C= 2 \frac{1}{\Gamma(\alpha/2+1)^{1-\alpha}} \left(\frac{1}{\Gamma(\alpha/2)} \frac{1}{1/2-\alpha/2}\right)^{\alpha}.
\end{align*}

\end{proof}

We next establish the validity of Lemma \ref{lemma2}
 
\begin{proof}[Proof of Lemma \ref{lemma2}]
We have
\begin{align*}
\| I_\alpha D\chi_E\|_{L^{d/(d-\alpha),1}(\mathbb{R}^d;\mathbb{R}^d)} &= \int_0^\infty |\{ |I_\alpha D\chi_E|>s\}|^{(d-\alpha)/d} \;ds \\
&=\int_0^1 |\{ |I_\alpha D\chi_E|>s\}|^{(d-\alpha)/d} \;ds + \int_1^\infty |\{ |I_\alpha D\chi_E|>s\}|^{(d-\alpha)/d} \;ds \\
&=: A+B.
\end{align*}
For $A$, we utilize the estimate \eqref{global}, the inequality
\begin{align*}
|\{\sup_{t>0} |p_t \ast D \chi_E|+\sup_{t>0} |t^{1/2}Dp_t \ast \chi_E|>\frac{s}{C}\}|&\leq  |\{\sup_{t>0} |p_t \ast D \chi_E|>\frac{s}{2C}\}|\\
&\;\;+|\{\sup_{t>0} |t^{1/2}Dp_t \ast \chi_E|>\frac{s}{2C}\}|,
\end{align*}
and the weak-type $(1,1)$ estimates for the Hardy-Littlewood maximal function (note Theorem 2 on p.~62 in \cite{Sharmonic} implies that both of these maximal functions associated to the heat kernel can be controlled above pointwise by a constant times the Hardy-Littlewood maximal function) to obtain
\begin{align*}
A &\leq C' \left(Per(E)+ |E|\right)^{1-\alpha/d} \int_0^1 \frac{1}{s^{1-\alpha/d}}\;ds \\
&=C'' \left(Per(E)+ |E|\right)^{1-\alpha/d}.
\end{align*}
Meanwhile, for $B$, the inequality \eqref{local}, Theorem 2 in \cite{Sharmonic}, and the weak-type $(1,1)$ estimate for the Hardy-Littlewood maximal function imply
\begin{align*}
B &\leq \int_1^\infty \left(\tilde{C} \frac{Per(E)}{s^{1/(1-\alpha)}}\right)^{(d-\alpha)/d} \;ds \\
&=\tilde{C}' Per(E)^{1-\alpha/d}.
\end{align*}
Putting these estimates together, and using subadditivity of the map $z \mapsto z^{1-\alpha/d}$, we find
\begin{align*}
\| I_\alpha D\chi_E\|_{L^{d/(d-\alpha),1}(\mathbb{R}^d;\mathbb{R}^d)} \leq C\left(Per(E)^{1-\alpha/d} +|E|^{1-\alpha/d}\right)
\end{align*}
The desired estimate then follows by an application of the inequality to the set $E_t:=\{x : tx \in E\}$ and an optimization in $t$.
\end{proof}

We next give the proof of Lemma \ref{thm2}.
\begin{proof}[Proof of Lemma \ref{thm2}]
First let us observe that a covering argument and concavity of the function (which holds because $\beta \in [1,d)$) $s \in [0,\infty) \mapsto s^{(d-\beta)/(d-1)}$ (see the proof of Corollary 1.5 and Remark 4.4 in \cite{PonceSpector2}) shows that
\begin{align*}
\left(\mathcal{H}^{d-1}_\infty \left(A\right)\right)^{(d-\beta)/(d-1)} \leq \mathcal{H}^{d-\beta}_\infty \left(A\right),
\end{align*}
and so it suffices to show the result for $\mathcal{H}^{d-1}_\infty$.

Thus we let $x=(x',x_d)$, $y=(y',y_d)$, and $\nu$ denote the unit exterior normal we first observe that
\begin{align*}
I_1 D\chi_Q(x) &= \int_{\partial Q} \frac{\nu(y)}{|x-y|^{d-1}}\;d\mathcal{H}^{d-1}(y) \\
&= \int_{\partial Q} \frac{\nu(y)}{\left(|x_d-y_d|^2+ |x'-y'|^2\right)^{(d-1)/2}}\;d\mathcal{H}^{d-1}(y)
\end{align*}
is smooth for $x \in \mathbb{R}^{d-1} \times \mathbb{R}^-$.  In particular, the fundamental theorem of calculus implies that for such $x$ one has
\begin{align*}
I_1(D\chi_Q)_d(x',x_d) &= \int_{\{\partial Q: y_d=1\} } \frac{1}{\left(|x_d-1|^2+ |x'-y'|^2\right)^{(d-1)/2}}\;d\mathcal{H}^{d-1}(y)\\
&\;\;-\int_{\{\partial Q: y_d=0\}} \frac{1}{\left(|x_d|^2+ |x'-y'|^2\right)^{(d-1)/2}}\;d\mathcal{H}^{d-1}(y) \\
&=\int_{Q'} \frac{1}{\left(|x_d-1|^2+ |x'-y'|^2\right)^{(d-1)/2}}\;dy'\\
&\;\;- \int_{Q'} \frac{1}{\left(|x_d|^2+ |x'-y'|^2\right)^{(d-1)/2}}\;dy',
\end{align*}
where we have used the notation $Q'=[0,1]^{d-1}$.  In particular for $x_d<0$ we have
\begin{align*}
\int_{Q'} \frac{1}{\left(|x_d-1|^2+ |x'-y'|^2\right)^{(d-1)/2}}\;dy' &\leq \int_{Q' }\int_{Q'} \frac{1}{|x_d-1|^{d-1}}\;dy' \\
&\leq |Q'|\\
&=1,
\end{align*}
and therefore
\begin{align*}
|I_1 D\chi_Q (x',x_d)| &\geq \int_{Q'} \frac{1}{\left(|x_d|^2+ |x'-y'|^2\right)^{(d-1)/2}}\;dy' - \int_{Q'} \frac{1}{\left(|x_d-1|^2+ |x'-y'|^2\right)^{(d-1)/2}}\;dy' \\
&\geq \int_{Q'} \frac{1}{\left(|x_d|^2+ |x'-y'|^2\right)^{(d-1)/2}}\;dy'-1.
\end{align*}

If $x' \in Q'$, then the change of variables $z'=\frac{x'-y'}{x_d}$ shows
\begin{align*}
\int_{Q'} \frac{1}{\left(|x_d|^2+ |x'-y'|^2\right)^{(d-1)/2}}\;dy' &= \int_{\frac{x'-Q'}{x_d}} \frac{1}{\left(1+ |z'|^2\right)^{(d-1)/2}}\;dz'\\
& \approx C\ln \frac{1}{|x_d|}
\end{align*}
for $|x_d|$ sufficiently small, and so in turn for such $x'$ we deduce the inequality
\begin{align*}
|I_1 D\chi_Q| \geq c\ln\left(\frac{1}{|x_d|}\right)
\end{align*}
for all $|x_d|$ sufficiently small.  

For $s \in (0,1)$, define $Q'_s:= [0,1]^{d-1}\times \{x_d=-s\}$ and $\mu_s:= \mathcal{H}^{d-1}\restriction Q'_s$.  Then P5 in \cite{Adams:1988} gives the inequality  
\begin{align*}
\mu_s(\{ |DI_1\chi_Q|>t\}) \leq C'\mathcal{H}^{d-1}_\infty(\{ |DI_1\chi_Q|>t\})
\end{align*}
for every $s \in (0,1)$.  But for each $t>0$, if $s<exp(-t/c)$ then
\begin{align*}
\{ |DI_1\chi_Q|>t\}\cap Q'_s = Q'_s
\end{align*}
and therefore
\begin{align*}
c'= \mathcal{H}^{d-1}\restriction Q'_s(Q'_s) \leq C'\mathcal{H}^{d-1}_\infty(\{ |DI_1\chi_Q|>t\}).
\end{align*}
Thus we have shown that
\begin{align*}
\frac{c'}{C'}t \leq t\mathcal{H}^{d-1}_\infty(\{ |DI_1\chi_Q|>t\}),
\end{align*}
and the claim of the Lemma follows.
\end{proof}

We conclude with the proof of Theorem \ref{cor}.
\begin{proof}[Proof of Theorem \ref{cor}]

Let $\{\rho_n\}_{n \in \mathbb{N}} \subset C^\infty_c(\mathbb{R}^d)$ be a sequence of standard mollifiers.  Then as $\chi_Q \ast \rho_n \in C^\infty_c(\mathbb{R}^d)$,   one has for every $x \in \mathbb{R}^d$ the equality
\begin{align*}
D I_1(\chi_Q \ast \rho_n) &= D I_{2-\alpha} ((-\Delta)^{(1-\alpha)/2}(\chi_Q \ast \rho_n)).
\end{align*}
From this formula and the semi-group property of the Riesz potentials and their inverses the fractional Laplacians, again using $\chi_Q \ast \rho_n \in C^\infty_c(\mathbb{R}^d)$, one obtains the estimate
\begin{align*}
|D I_1(\chi_Q \ast \rho_n)| &\leq c I_{1-\alpha} |(-\Delta)^{(1-\alpha)/2}(\chi_Q \ast \rho_n)|.
\end{align*}

For $s \in (0,1)$, define $Q'_s:= [0,1]^{d-1}\times \{x_d=-s\}$ and $\mu_s:= \mathcal{H}^{d-1}\restriction Q'_s$.  Then as $D I_1(\chi_Q \ast \rho_n), I_{1-\alpha} |(-\Delta)^{(1-\alpha)/2}(\chi_Q \ast \rho_n)|$ are continuous function on $Q'_s$, we obtain by integration
\begin{align*}
\int  |D I_1(\chi_Q \ast \rho_n)| \;d\mu_s \leq c \int I_{1-\alpha}|(-\Delta)^{(1-\alpha)/2}(\chi_Q \ast \rho_n)| \;d\mu_s.
\end{align*}
Next Fubini's theorem implies
\begin{align*}
\int  |D I_1(\chi_Q \ast \rho_n)| \;d\mu_s \leq c \int |(-\Delta)^{(1-\alpha)/2}(\chi_Q \ast \rho_n)| I_{1-\alpha}\mu_s.
\end{align*}
Now the fact that $\mu_s(B(x,r)) \leq C'r^{d-1}$) implies, by the proof of Proposition 5 on p.~121 in \cite{Adams:1988}, that  $\mu=I_{1-\alpha}\mu_s$ is a measure that satisfies the growth condition $\mu(B(x,r)) \leq C' r^{d-\alpha}$ (in D. Adams' paper, he would write $\mu_s \in L^{1,d-1}_+$ implies $I_{1-\alpha}\mu_s \in L^{1,d-\alpha}$).  Therefore if one had the trace inequality
\begin{align*}
\int_{\mathbb{R}^d} |u| \;d\mu \leq C\int_{\mathbb{R}^d} |D^\alpha u|
\end{align*}
for $u \in L^q(\mathbb{R}^d)\cap C(\mathbb{R}^d)$ for some $1\leq q <d/(1-\alpha)$ such that $D^\alpha u \in L^1(\mathbb{R}^d;\mathbb{R}^d)$ one would deduce
\begin{align*}
\int  |D I_1(\chi_Q \ast \rho_n)| \;d\mu_s \leq C' \int_{\mathbb{R}^d} |D(\chi_Q \ast \rho_n)|.
\end{align*}
Notice we have here utilized the equality $D^\alpha (-\Delta)^{(1-\alpha)/2}(\chi_Q \ast \rho_n) = D (\chi_Q \ast \rho_n)$.  

However, as 
\begin{align*}
\int_{\mathbb{R}^d} |D(\chi_Q \ast \rho_n)| \leq |D\chi_Q|(\mathbb{R}^d),
\end{align*}
this would imply the inequality
\begin{align*}
\int  |D I_1(\chi_Q \ast \rho_n)| \;d\mu_s \leq C' |D\chi_Q|(\mathbb{R}^d).
\end{align*}
As we let $n$ tend to infinity, from the computation in the proof of Lemma \ref{thm2} we see that on the set $Q'_s$ one has the uniform convergence of the continuous functions $D I_1(\chi_Q \ast \rho_n)$ to the continuous function $D I_1\chi_Q$.  Therefore Lebesgue's dominated convergence theorem would yield
\begin{align*}
\int  |D I_1\chi_Q| \;d\mu_s \leq C' |D\chi_Q|(\mathbb{R}^d),
\end{align*}
and Chebychev's inequality would, in turn, imply that for any $t>0$
\begin{align*}
t \mu_s(\{|D I_1\chi_Q|>t\}) \leq  C' |D\chi_Q|(\mathbb{R}^d).
\end{align*}
However, the proof of Lemma \ref{thm2} shows that we can make the left-hand-side as large as we like by choosing $s,t$ appropriately, and the result is demonstrated.
\end{proof}

\section*{Acknowledgements}
The author would like to thank the referees for their careful reading and many comments that have greatly improved the paper.  Needless to say the author is responsible for the remaining shortcomings.  The author is supported in part by the Taiwan Ministry of Science and Technology under research grant 107-2115-M-009-002-MY2.


\bibliographystyle{amsplain}


\end{document}